\documentclass[12pt]{amsart}
\usepackage{amsmath}
\usepackage{amsthm}
\usepackage{amssymb}
\usepackage{enumerate}
\usepackage{comment}
\newcommand{\n}{\par\noindent}

\newcommand{\sn}{\par\smallskip\noindent}
\newcommand{\mn}{\par\medskip\noindent}
\newcommand{\bn}{\par\bigskip\noindent}
\newcommand{\pars}{\par\smallskip}
\newcommand{\parm}{\par\medskip}

\newtheorem{theorem}{Theorem}
\newtheorem{lemma}[theorem]{Lemma}

\newtheorem{corollary}[theorem]{Corollary}

\newtheorem{remar}[theorem]{Remark}

\newcommand{\gzaun}{\unskip\nobreak\hfil\penalty50%
\hskip1em\hbox{}\nobreak\hfil%
$\#$\parfillskip=0pt\finalhyphendemerits=0}

\newcommand{\bfind}[1]{\index{#1}{\bf #1}}

\newcommand{\cO}{\mathcal{O}}
\newcommand{\cM}{\mathcal{M}}

\newcommand{\N}{\mathbb N}

\begin{document}            

\title[]{Large imperfect fields are existentially closed in function fields after
finite constant extension}


\author{Hagen Knaf}
\author{Franz-Viktor Kuhlmann}

\thanks{The authors thank Sylvy Anscombe and Arno Fehm for their valuable input.
The second author also wishes to thank the Hausdorff Research Institute for Mathematics
in Bonn for support and hospitality, as his stay at the Institute during the program
``Definability, decidability, and computability'' in 2025 gave
him the opportunity to discuss the contents of this paper with several colleagues.}

\date{31.\ 5.\ 2026}

\begin{abstract}
For an algebraic function field $F$ over a large field $K$, we show: \n
1) if $F|K$ has a rational place, then there is a finite purely inseparable extension
$K'|K$ such that $K'$ is existentially closed in $F.K'$; \n
2) $F|K$ has a rational place admitting local uniformization if and only if
$K$ is existentially closed in $F$.
\end{abstract}

\maketitle

%
%
\section{Introduction}
For notions and notation used in this introduction, see Section~\ref{sectprel}. In
what follows, by ``function field'' we always mean ``algebraic function field''.

In \cite{Ku20}, the following question is considered: if $F|K$ is a function field that
admits a rational place, under which additional assumptions does it follow that $K$ is
existentially closed in $F$ in the language of rings? The following statement is part
of \cite[Theorem~17]{Ku20}:
\begin{theorem}                             \label{KecF1}
If $K$ is a perfect large field, then $K$ is existentially closed in every extension
field $L$ which admits a $K$-rational place.
\end{theorem}

If local uniformization in positive characteristic holds, then by
\cite[Proposition~2.3]{ADF} the condition
``perfect'' in Theorem~\ref{KecF1} can be dropped and we obtain a principle that has
been introduced in \cite{ADF} as a basic condition for the main result in that article,
weaker than local uniformization:
\sn
{\bf (R4)}\; If $K$ is a large field, then $K$ is existentially closed in every
extension field $L$ which admits a $K$-rational place.

\pars
In this note our goal is to prove the following result. Our proof in
Section~\ref{sectprfs} is a slight modification of a proof suggested to us by Arno Fehm.
Likewise, also Corollary~\ref{corratpl} and Open Problem 2 are due to him.
\begin{theorem}                              \label{MT}
Take a large field $K$ and a function field $F|K$ which admits a rational place. Then
there is a finite purely inseparable extension $K'|K$ such that $K''$ is existentially
closed in the compositum $F.K''$ for every algebraic extension $K''|K'$.
\end{theorem}

Hence we can say: {\it if $L|K$ is a function field, then (R4) holds after a finite purely
inseparable constant extension}.
\sn
If $L|K$ is not finitely generated, then for arbitrary elementary sentences valid in
$L$, ever larger purely inseparable extensions of $K$ may be needed. However,
Theorem~\ref{MT} implies Theorem~\ref{KecF1}.

\pars
The following is part of \cite[Lemma 9]{F}:
\begin{lemma}
Take a large field $K$ and a function field $F|K$. Then $F|K$ admits a discrete rational
place if and only if $K$ is existentially closed in $F$.
\end{lemma}

Therefore, Theorem~\ref{MT} is equivalent to:
\begin{theorem}                              \label{MT'}
Take a large field $K$ and a function field $F|K$ which admits a rational place. Then
there is a finite purely inseparable extension $K'|K$ such that for every algebraic extension $K''$ of $K'$, the function field $F.K''|K''$ admits a discrete
rational place.
\end{theorem}

Note that if $F.K'|K'$ admits a discrete rational place $P$, then $FP$ is a finite
purely inseparable extension of $K$. Hence one is led to:
\sn
{\bf Open problem 1:}  Assume that $F|K$ admits a discrete place $P$ whose residue field
is a finite purely inseparable extension of $K$. Under which additional assumptions does
it follow that there is a finite purely inseparable extension $K'|K$ such that the unique
extension of $P$ to $F.K'$ has residue field $K'$?

\sn
It is known that $K'$ may not exist for a function field $F$ over a field $K$ of
characteristic $p>0$ and infinite degree of imperfection. Indeed, take a sequence
$(a_i)_{i\in\N}$ in $K$ of $p$-independent elements, set $s=\sum_na_it^{ip}\in K[[t]]$,
and take $F=K(t,s^{1/p})$. Then the residue field of $F$ under the $t$-adic valuation
is $K_1=K(a_0^{1/p})$, but extending the constant field to $K_1$ will yield
residue field $K_2=K_1(a_1^{1/p})$, and so on. This is a construction given by Masayoshi
Nagata in \cite[Appendix, Example~(E3.1), pp.~206-207]{Na}, cf.\ \cite[Example 3.23
and Theorem 3.24]{Ku31}.

We do not know whether there is such
an example with a field $K$ of finite degree of imperfection.

\parm
If we want to avoid the purely inseparable extension appearing in Theorem \ref{MT}, we
have to invoke local uniformization, which presently is not known to hold in positive
characteristic for dimensions higher than three.
\begin{theorem}                             \label{MT2}
For a function field $F$ over a large field $K$ the following conditions are
equivalent:
\begin{enumerate}

\item There exists a rational place admitting local uniformization.

\item $K$ is existentially closed in $F$.

\end{enumerate}
\end{theorem}
The implication $(1)\Rightarrow (2)$ has already been stated in \cite[Theorem 19]{Ku20}.
However, its proof uses \cite[Theorems 13 and 14]{Ku20}, but the proof of Theorem 13 is
flawed. In order to repair it, one might attempt to use our new proof of implication
$(1)\Rightarrow (2)$ together with the fact that the large field $K$ being
existentially closed in $F$ implies that $F|K$ admits a discrete rational place $Q$
(see \cite[Lemma 9]{F}). However, Theorem 13 claims more, namely that $Q$ can be chosen
as close to $P$ as we want in the patch topology.

%
%
\section{Preliminaries}                          \label{sectprel}
For basic facts from valuation theory, see \cite{EP}.

\pars
Let $L|K$ be a field extension. We will say that $K$ is \bfind{existentially closed} in
$L$ if every existential sentence in the ring (or field) language with parameters
from $K$ holds in $K$ if it holds in $L$.

\pars
Take a field $K$. A \bfind{$K$-variety} is an integral separated $K$-scheme $V$ of
finite type; if $V$ has dimension one, then it is called a \bfind{$K$-curve}. As usual
the function field of a $K$-variety $V$ is denoted by $K(V)$ and for any extension field
$L$ of $K$ the set of $L$-rational points of $V$ is denoted by $V(L)$. The subset
$V_\mathrm{sm}$ of smooth points of $V$, called the \bfind{smooth locus of $V$}, is
Zariski-open and thus a $K$-variety if it is nonempty (see \cite[Corollary 8.2]{Kunz});
the latter holds if and only if $K(V)|K$ is separable, i.e., linearly disjoint from
$K^{1/p^\infty}|K$, where $K^{1/p^\infty}$ denotes the perfect hull of $K$. If
$V=V_\mathrm{sm}$, then the $K$-variety $V$ is said to be \bfind{smooth}.

\pars
A $K$-variety $V$ is \bfind{geometrically integral} if the scheme $V\times_K
\widetilde{K}$ is a $\widetilde{K}$-variety, where $\widetilde{K}$ denotes the
algebraic closure of $K$. According to \cite[Corollaire 4.6.3]{G}, this is the case
if  and only if the extension $K(V)|K$ is \bfind{regular}, that is, $K(V)|K$ and
$\widetilde{K}|K$ are linearly disjoint. This in turn holds if and only if
$K$ is relatively algebraically closed in $K(V)$ and $K(V)|K$ is separable (cf.\ \cite[Chapter VIII, \S4]{La}).

\pars
A \bfind{model} of a given function field $F|K$ is a $K$-variety $V$ with the property
$K(V)\cong_K F$. A place $P$ of a function field $F|K$ is called \bfind{rational} if it
is the identity on $K$ and $FP=K$. The existence of a rational place $P$ on $F|K$
implies the existence of a model $V$ possessing a $K$-rational point. This is seen as
follows. Take a transcendence base $T$ in the valuation ring $\cO_P$ of the place $P$
and denote by $A$ the integral closure of $K[T]$ in $\cO_P$. Then $A$ is finitely
generated over $K$ and $V:=\mathrm{Spec}(A)$ is a model with the $K$-rational point
$x:=\cM_P\cap A$, where $\cM_P$ denotes the maximal ideal of $\cO_P\,$. However, in
general such a point $x$ need not be regular.

\pars
If there is a model $V$ for which the center of $P$ on $V$
is regular, we say that $P$ admits \bfind{local uniformization}.

\pars
Following Florian Pop (see \cite{P1,P2}), a field $K$ is called a large field if it
satisfies one of the following equivalent conditions:
\sn
{\bf (LF)} \ {\it For every smooth $K$-curve $C$ the set $C(K)$ is infinite or empty.}
\sn
{\bf (LF$'$)} \ {\it In every smooth $K$-variety $V$ the set $V(K)$ is Zariski-dense or
empty.}
\sn
{\bf (LF$''$)} \ {\it For every function field $F|K$ in one variable
the set of rational places is infinite or empty.}

\parm
Various characterizations of large fields using existential closedness are given by
\cite[Theorem 15]{Ku20}:
\begin{theorem}                             \label{lfecKt}
The following conditions are equivalent:
\sn
1) \ $K$ is a large field,
\sn
2) \ $K$ is existentially closed in every function field $F$ in one
variable over $K$ which admits a $K$-rational place,
\sn
3) \ $K$ is existentially closed in the henselization $K(t)^h$ of the
rational function field $K(t)$ with respect to the $t$-adic valuation,
\sn
4) \ $K$ is existentially closed in the field $K((t))$ of formal Laurent
series,
\sn
5) \ $K$ is existentially closed in every extension field which admits a
discrete $K$-rational place.
\end{theorem}

\mn
%
%
\section{Proofs}                          \label{sectprfs}
One of the main ingredients in the proofs of Theorem \ref{MT} and Theorem \ref{MT2} is
a geometric characterization of existential closedness of a field within a finitely
generated extension, due to Pop \cite[Fact 2.3]{P2}:
\begin{lemma}                                     \label{lem}
Let $K$ be a field and $V$ a $K$-variety. Then $K\prec_\exists K(V)$
if and only if $V(K)$ is Zariski-dense in $V$.
\end{lemma}
The following result is well known, but the only reference we know is
\cite[Corollary 3.1.3]{E}, where it is derived from a more involved context. For the
convenience of the reader, we will give a direct proof.
\begin{lemma}                                     \label{lem2}
Let $F|K$ be a function field with the property $K\prec_\exists F$,
then $F|K$ is a regular extension.
\end{lemma}
\begin{proof}
Suppose that $K$ is not relatively algebraically closed in $L$.
Then there exists an irreducible polynomial $f(X)=X^n+c_{n-1}X^{n-1}
+\ldots +c_0\in K[X]$ which has a root in $L$ but not in $K$. The existential
sentence
\[
\exists x:\> x^n+c_{n-1}x^{n-1}+\ldots+c_0\>=\>0
\]
with parameters in $K$
expresses the existence of a root of $f$. By our choice of $f$, it holds
in $L$ but not in $K$. Consequently, $K$ is not existentially closed in $L$.

Now suppose that $L|K$ is not separable, i.e., not linearly disjoint from
$K^{1/p^{\infty}}|K$.
Then there are $K$-linearly independent elements $x_1,\ldots,x_n\in L$
which are not $K^{1/p^{\infty}}$-linearly independent. Choose $m$
minimal such that there are $y_1,\ldots,y_n\in K^{1/p^m}$ with $\sum_{i}
x_iy_i=0$. Then $m\geq 1$, and by the minimality of $m$, $x_1^{p^{m-1}},\ldots,
x_n^{p^{m-1}}$ are $K$-linearly independent.
But $x_1^{p^{m-1}},\ldots,x_n^{p^{m-1}}$ are not $K^{1/p}$-linearly
independent since $\sum_{i} x_i^{p^{m-1}} y_i^{p^{m-1}}=0$ with
$y_i^{p^{m-1}}\in K^{1/p}$. This proves that $L|K$ is not linearly
disjoint from $K^{1/p}|K$.

The Frobenius induces an isomorphism from $L$ onto $L^p$, from $K^{1/p}$ onto $K$, and
from $K$ onto $K^p$. Hence it follows that $L^p|K^p$ is not linearly disjoint
from $K|K^p$. Choose $K^p$-linearly independent elements $a_1,\ldots,a_n\in K$ which
are not $L^p$-linearly independent. Then the existential
sentence
\[
\exists x_1\,\exists x_2\,\ldots\,\exists x_n:\>
x_1^p a_1+\ldots+x_n^p a_n\>=\>0
\]
with parameters in $K$ holds in $L$ but not in $K$, showing again that $K$ is not
existentially closed in $L$.
\end{proof}

\cite[Lemma 2.6.9]{FJ} provides another criterion for regularity relevant in the present
context:
\begin{lemma}                         \label{ratpl}
Let $F|K$ be a function field which admits a rational place, then $F|K$ is
regular.
\end{lemma}

\sn
\begin{proof}[Proof of Theorem \ref{MT}]
By Lemma~\ref{ratpl}, $F|K$ is a regular extension. Hence there exists a geometrically
integral, smooth model $V$ of $F|K$. One can assume that $V$ is affine:
$V=\mathrm{Spec}(A)$.

The place $P$ has a unique extension $P'$ to $F':=F.K^{1/p^\infty}$, and $F'P'=
K^{1/p^\infty}$. By \cite[Prop.~2.7]{P2} $K^{1/p^\infty}$ is a large field. Thus, by
Theorem~\ref{KecF1}, $K^{1/p^\infty}\prec_\exists F'$. The variety $W:=V\times_KK^{1/
p^\infty}=\mathrm{Spec}(A\otimes_KK^{1/p^\infty})=\mathrm{Spec}(A\cdot K^{1/p^\infty})$
is a model of $F'|K^{1/p^\infty}$. Lemma \ref{lem} then yields that
$W(K^{1/p^\infty})$ is Zariski-dense in $W$.

Fix any $q\in W(K^{1/p^\infty})$ and consider the point $p:=\phi(q)$, where
$\phi:W\rightarrow V$ is the natural projection. Then $K':=A/p$ is a finite
purely inseparable extension of $K$. Now let $K''|K'$ be any algebraic extension,
$W'':=V\times_KK''=\mathrm{Spec}(A\cdot K'')$ and $\phi'':W''\rightarrow V$ the
natural projection. Then any $q''\in W''$ such that $\phi''(q'')=p$ satisfies
$(A\cdot K'')/q''=A/p\cdot K''=K''$, that is, $q''\in W''(K'')$. Since $W''$ is a
smooth $K''$-variety and $K''$ is large by \cite[Prop.~2.7]{P2},
$W''(K'')$ is Zariski-dense by {\bf (LF$'$)}. Lemma~\ref{lem}
now yields $K''\prec_{\exists}F.K''$.
\end{proof}

\begin{corollary}                         \label{corratpl}
For $F|K$ as in Theorem~\ref{MT}, there exists $n\in\mathbb{N}$ such that
$K\prec_\exists F^{p^n}K$, where $p={\rm char}(K)$.
\end{corollary}

\begin{proof}
If $K'$ is as in the theorem,
then $K'\subseteq K'':=K^{p^{-n}}$ with $p^n=[K':K]$. It follows that $K''\prec_\exists
F.K''$, and therefore $K=(K'')^{p^n}\prec_\exists (F.K'')^{p^n}=F^{p^n}.K$.
\end{proof}

\sn
{\bf Open problem 2:} Are Theorem~\ref{MT} and Corollary~\ref{corratpl} also true if
$F|K$ is not a function field, i.e., not finitely generated?

\begin{proof}[Proof of Theorem \ref{MT2}]
Assume that the rational place $P$ of the function field $F|K$ admits local
uniformization. Then there exists a model $V$ of $F|K$ such that $P$ is centered in a
regular, $K$-rational point $x\in V$. By Lemma \ref{ratpl}, the existence of a rational
place shows that the extension $F|K$ is separable. Therefore the smooth locus of $V$ is
Zariski-open, so we can assume $V$ to be smooth. A regular, $K$-rational point is
smooth, hence $V(K)\neq\emptyset$. Since $K$ is large, $V(K)$ is dense in $V$, thus
$K\prec_\exists F$ by Lemma \ref{lem}.

Now assume that $K\prec_\exists F$. Lemma \ref{lem2} yields that $F|K$ is separable.
Therefore the smooth locus $V_\mathrm{sm}$ of any model $V$ of $F|K$ is nonempty
and Zariski-open. By Lemma \ref{lem}, $V(K)$ is dense in $V$, which implies $V_\mathrm{sm}
\cap V(K)\neq\emptyset$. Take $x\in V_\mathrm{sm}\cap V(K)$. By \cite[Cor.\ A2]{JR} (see
also \cite[Lemma 16.1.2]{FJ}) there exists a $K$-rational place $P$ centered in $x$ and
since $x$ is regular, $P$ by definition admits local uniformization.
\end{proof}

\bn\bn\bn

\end{document}